\documentclass{article}

\usepackage{hyperref}
\usepackage{amsfonts}
\usepackage{amsmath}
\usepackage{amsthm}
\usepackage{amssymb}
\usepackage[ruled]{algorithm2e}
\usepackage{authblk}

\newtheorem{theorem}{Theorem}
\newtheorem{proposition}{Proposition}

\theoremstyle{remark}
\newtheorem{remark}{Remark}
\theoremstyle{definition}
\newtheorem{example}{Example}[section]

\begin{document}

		\title{Counting points on hyperelliptic curves of type $y^2=x^{2g+1} + ax^{g+1} + bx$%
		}

		\author{Novoselov~S.\,A.\footnote{The reported study was funded by RFBR according to the research project 18-31-00244.}}
		\affil{Immanuel Kant Baltic Federal University}
		
		\date{}
		
		\maketitle

		\begin{abstract}
			In this work, we investigate hyperelliptic curves of type $C: y^2 = x^{2g+1} + ax^{g+1} + bx$ over the finite field $\mathbb{F}_q, q = p^n, p > 2$. For the case of $g = 3$ and $4$ we propose algorithms to compute the number of points on the Jacobian of the curve with complexity $\tilde{O}(\log^4{p})$ and $\tilde{O}(\log^8{p})$. For curves of genus $2-7$ we give a complete list of the characteristic polynomials of Frobenius endomorphism modulo $p$.
		\end{abstract}

\section*{Introduction}
Let $\mathbb{F}_q$ be a finite field of size $q = p^n$ and characteristic $p > 2$.

Let $C$ be a genus $g$ hyperelliptic curve defined over $\mathbb{F}_q$ by equation
\[
y^2 = x^{2g+1} + a x^{g+1} + b x.
\]

For genus $2$ case it is known \cite{Satoh2009} that the Jacobian $J_C$ of the curve $C$ splits into product of certain explicitly given elliptic curves over some extension of base field. There are also explicit formulae \cite{GuillevicVergnaud2012} expressing the number of points on the Jacobian of the curve in terms of traces of Frobenius of the elliptic curves.

The purpose of this note is to generalize results for $g = 2$ to the higher genera ($g \geq 3$) and derive algorithms for counting points on $J_C$ in this case. To speed up point counting, where it is possible, we use connection of the Cartier-Manin matrix of the curve $C$ with Legendre polynomials from \cite{Novoselov2017}. On the other side, this connection is also used to obtain new congruences for Legendre polynomials $P_{\frac{p-1}{8}}, P_{\frac{3p-3}{8}}$ extending the results from the works \cite{BrillhartMorton2004,Sun2013,Sun2013b,Sun2013c}.

\textit{Background and notation.}
Let $X$ be a genus $g$ hyperelliptic curve defined over finite field $\mathbb{F}_q$. A zeta-function of the curve $X$ is given by
\[
Z(X/\mathbb{F}_q; T) = \exp\left( \sum_{k=1}^{\infty} \frac{\#X(\mathbb{F}_{q^k})}{k} T^k \right)
= \frac{L_{X,q}(T)}{(1-T)(1-qT)}.
\]
The polynomial $L_{X,q}(T)$ has a form
\[
L_{X,q}(T) = q^g T^{2g} + a_1 q^{g-1} T^{2g-1} + ... + a_g T^{g} + a_{g-1} T^{g-1} + ... + a_1 T + 1,
\]
where $a_i \in \mathbb{Z}$ and $a_i \leq \binom{2 g}{i} q^{\frac{i}{2}}$. Coefficients of the polynomial $L_{X,q^k}(T)$ are denoted by $a_{i,k}$.

Let $J_X(\mathbb{F}_q)$ be the Jacobian of the curve over finite field $\mathbb{F}_q$ and let $\chi_{X,q}(T)$ be the characteristic polynomial of the Frobenius endomorphism on $J_X$. Then we have $L_X(T) = T^{2g} \chi_{X,q}(\frac{1}{T})$ and $\#J_X(\mathbb{F}_q) = L_{X,q}(1) = \chi_{X,q}(1)$. Because of that, point-counting on the Jacobian is equivalent to determining of $\chi_{X,q}(T)$ (or $L_{X,q}(T)$).

The order of the Jacobian satisfies Hasse-Weil bounds
\[
(\sqrt{q} - 1)^{2g} \leq \#J_X(\mathbb{F}_q) \leq (\sqrt{q} + 1)^{2g}.
\]
For more details on curves over finite fields and their Jacobians we refer to \cite[\S5.2]{CohenFrey+2005}.

\textit{Organization of the paper.}
In Section \ref{sec:jacobian_decomposition} we obtain a decomposition of $J_C \sim J_{X_1} \times J_{X_2}$ over $\mathbb{F}_q[\sqrt[g]{b}]$ for odd $g$ and over $\mathbb{F}_q[\sqrt[2g]{b}]$ in case of $g$ is even. For curves $X_1$ and $X_2$, we provide explicit equations in terms of Dickson polynomials. This is done by using results and methods from \cite{Paulhus2007,Paulhus2008,TauTopVer1991,Smith2005}.

Section \ref{sec:CM_matrices} is devoted to the study of Cartier-Manin matrix of the curve $C$. We obtain a complete list of possible polynomials $\chi_{C,p}(T) \mod{p}$. This fills missing cases in our previous work \cite{Novoselov2017}. We use this list to derive a point-counting algorithm in genus $3$ case and to find new congruences for Legendre polynomials connected with genus $4$ curves.

Section \ref{sec:chi_C_computation} contains description of a general method for counting points on $J_C(\mathbb{F}_q)$ using decomposition from Section \ref{sec:jacobian_decomposition} and results from Section \ref{sec:CM_matrices}. Sections \ref{sec:pc_genus3}, \ref{sec:pc_genus4} contain algorithms and implementation details for $g=3,4$.

\section{Decomposition of the $J_C$.}
\label{sec:jacobian_decomposition}
The curve $C$ is isomorphic to curve
\[
C': y^2 = x^{2g+1} + c x^{g+1} + x, c = \pm \frac{a}{\sqrt{b}}
\] over finite field $\mathbb{F}_q[\sqrt[4g]{b}]$ via isomorphism $(x,y) \mapsto (b^{\frac{1}{2g}} x, b^{\frac{2g+1}{4g}} y)$. Because of that, we can get decomposition for $J_{C}$ over $\mathbb{F}_q[\sqrt[4g]{b}]$ from decomposition of $J_{C'}$. Since the curve $C'$ has automorphisms, we can decompose $J_{C'}$ by using method of Kani and Rosen\cite{KaniRosen1989}. In the work of Paulhus\cite{Paulhus2007} there is decomposition for the curve $C'$  over algebraically closed field. But the method works over any field as long as we know the group of automorphisms or its subgroups. Thus, we need to obtain information about subgroups of this group over finite field.

Denote by $Aut_{C}(\mathbb{F}_q)$ an automorphism group of curve over the finite field $\mathbb{F}_q$, $\mathcal{C}_m$ a cyclic group of order $m$ and by $\mathcal{D}_{m}$ a dihedral group of order $m$. Let also $\zeta_m$ be a primitive $m$-root of unity. Every hyperelliptic curve has hyperelliptic involution, denote it by $\omega: (x,y) \mapsto (x,-y)$. 

All possible groups of automorphisms for hyperelliptic curve over algebraically closed field are known \cite{BrandtStichtenoth1986,BujalanceGamboaGromadzki1993}. In \cite{TauTopVer1991} there are explicit automorphisms for curve $C'$. We collect them in the following proposition.

\begin{proposition}
	\label{th:gen_aut_group_of_c_rho}
	Let $C': y^2 = x^{2g+1} + c x^{g+1} + x$ be a genus $g$ hyperelliptic curve defined over a finite field $\mathbb{F}_q, q = p^n$.
	\begin{enumerate}
		\label{th:gen_aut_nonhyp_involution_s}
		\item $Aut_{C'}(\mathbb{F}_q)$ contains a non-hyperelliptic involution
		\[
		s: (x,y) \mapsto \left(\frac{1}{x}, \frac{y}{x^{g+1}}\right)
		\] and subgroup $\mathcal{C}_2 \times \mathcal{C}_2$.
		\item If $p \nmid 2g$ and $2g | q-1$ then $Aut_{C'}(\mathbb{F}_q)$ contains an automorphism
		\[
		r: (x, y) \mapsto \left( \zeta_{g} x, \zeta_{2g} y \right)
		\] of order $2g$ and subgroup $\mathcal{D}_{4g}$.
	\end{enumerate}
\end{proposition}

Decomposition for the Jacobian of the curve $C'$ in case of $g = 3,4$ follows from Proposition \ref{th:gen_aut_group_of_c_rho} and \cite[Th. 4]{Paulhus2008}:
\begin{itemize}
\item (genus 3)
$J_{C'} \sim E \times J_{X'}$ if $\mathcal{C}_{2} \times \mathcal{C}_{2} \subseteq Aut_{\mathbb{F}_q}(C')$ and $J_{C'} \sim E_1^2 \times E_2$ if $\mathcal{D}_{12} \subseteq Aut_{\mathbb{F}_q}(C')$
\item (genus 4) $J_{C'} \sim J_{X'_1} \times J_{X'_2}$ if $\mathcal{C}_{2} \times \mathcal{C}_{2} \subseteq Aut_{\mathbb{F}_q}(C')$ and $J_{C'} \sim J_{X'_1}^2$ if $\mathcal{D}_{16} \subseteq Aut_{\mathbb{F}_q}(C')$
\end{itemize}

Note that condition $\mathcal{C}_{2} \times \mathcal{C}_{2} \subseteq Aut_{\mathbb{F}_q}(C')$ holds in any field, therefore we always have corresponding decomposition. But $\mathcal{D}_{4g} \subseteq Aut_{C'}$ holds in the field $\mathbb{F}_q[\zeta_{2g}]$, so we should work in an extension of $\mathbb{F}_q$ of degree upto $2g$ to get decomposition. The degree of this extension is the smallest integer $k$ such that $2g | q^k-1$.

Since the group of automorphisms contains $\mathcal{C}_2 \times \mathcal{C}_2$ the Jacobian splits as
\[
J_{C'}(\mathbb{F}_q) \sim J_{C'/\left< s \right>}(\mathbb{F}_q) \times J_{C'/\left< s \omega \right>}(\mathbb{F}_q).
\]

To find equations for quotients of the curve $C'$, we can use the following theorem.

\begin{theorem}
\label{th:quotients_of_Ct}
Let $C': y^2 = x^{2g+1} + c x^{g+1} + x$ be a genus $g$ hyperelliptic curve defined over a finite field $\mathbb{F}_q$ where $q = p^n$, $p>2$, $\gcd(g, p) = 1$. Denote by $D_{m}(x, \alpha)$ and $D_{m}:=D_m(x, 1)$ a Dickson polynomial of degree $m$. Then 
\begin{enumerate}
	\item the quotient of the curve $C'$ modulo the involution $s: (x,y) \mapsto \left(\frac{1}{x}, \frac{y}{x^{g+1}}\right)$ is given by
	\[
		X'_1: y^2 = D_{g}(x) + c
	\]
	if $g$ is odd and by
	\[
		X'_1: y^2 = (x+2)(D_{g}(x) + c),
	\]
	if $g$ is even.
	\item the quotient of the curve $C'$ modulo the involution $s \omega: (x,y) \mapsto \left(\frac{1}{x}, -\frac{y}{x^{g+1}}\right)$ is given by
	\[
		X'_2: y^2 = (x^2 - 4) (D_{g}(x) + c)
	\]
	if $g$ is odd and by
	\[
	X'_2: y^2 = (x-2)(D_{g}(x) + c),
	\]
if $g$ is even.
\end{enumerate}
\end{theorem}
\begin{proof}
1. From \cite[Prop.3]{TauTopVer1991} we have that the quotient $C'/\left<s\right>$ is given by 
\[
y^2 = x f(x^2 - 2) + c
\]
if $n$ is odd and
\[
y^2 = (x+2) (f(x^2 - 2) + c)
\]
in case of $g$ is even. The polynomial $f \in \mathbb{F}_q$ is the monic polynomial whose roots are all the numbers $\zeta + \frac{1}{\zeta}$ for $\zeta \in \mathbb{\overline{F}}_q, \zeta^g = -1, \zeta \neq -1$.

In the work \cite[\S7.3]{Smith2005} it was shown that for odd $g$ we have $x f(x^2 - 2) = D_g(x)$.

The result for even $g$ follows from the factorization of the Dickson polynomial \cite[Th.1]{BhargavaZieve1999} over $\mathbb{\overline{F}}_q$
\[
D_g(x, \alpha) = \prod_{k=1,\,k\,odd}^{2 g - 1} (x - \sqrt{\alpha} (\zeta_{4 g}^k + {\zeta_{4 g}^{-k}})) = f(x^2 - 2), 
\]
where $\zeta_{4 g}$ is a primitive $4g$-root of unity.

2. The proof is similar to the proof in \cite[Prop.2]{TauTopVer1991} with addition of Dickson polynomials. We use relations
\[
x^{2g} + 1 = x^g \left(x+\frac{1}{x}\right) f\left(x^2 + \frac{1}{x^2}\right) = x^{g-1} \left(x+\frac{1}{x}\right) D_g\left( x + \frac{1}{x} \right)
\] for odd $g$ and
\[
x^{2g} + 1 = x^g  f\left(x^2 + \frac{1}{x^2}\right) = x^{g} D_g\left( x + \frac{1}{x} \right)
\] for even $g$.

A function field of $C'/\left<s\omega\right>$ is a field of $s\omega$-invariant functions from function field of $C'$. It's generated by $\xi = \left( x + \frac{1}{x} \right)$, $\eta = \frac{y}{x^{\frac{g-1}{2}}}  \left(1 - \frac{1}{x^{2}}\right)$ in the case of $g$ odd and by  $\xi = \left( x + \frac{1}{x} \right), \eta = \frac{y}{x^{\frac{g}{2}}} \left(1 - \frac{1}{x}\right)$ when $g$ is even.

Using the relations and a property of Dickson polynomials:
\[
	D_g(x+1/x) = x^g + 1/x^g
\]
we find an equation
\[
\eta^2 = (\xi^2 - 4) (D_g(\xi) + c)
\]
for odd $g$ and
\[
\eta^2 = (\xi - 2) (D_g(\xi) + c)
\] for even $g$.
\end{proof}

Theorem \ref{th:quotients_of_Ct} allows us to find explicit equations for quotient curves in the decomposition of $C'$ over finite field $\mathbb{F}_q$. Since $C \simeq C'$ we obtain decomposition for the curve $C$:
\[
J_C(\mathbb{F}_q[\sqrt[4g]{b}]) \simeq J_{C'/\left< s\right>}(\mathbb{F}_q[\sqrt[4g]{b}]) \times J_{C'/\left< \omega s\right>}(\mathbb{F}_q[\sqrt[4g]{b}]).
\]

The theorem can be generalized to the more general class of curves using results from \cite[\S7.3]{Smith2005}. It gives us better decomposition for $J_C$, which occurs over the field $\mathbb{F}_q[\sqrt[g]{b}]$ in the case of odd $g$ and over $\mathbb{F}_q[\sqrt[2 g]{b}]$ when $g$ is even.

\begin{theorem}
\label{th:X_quotients}
Let $X: y^2 = x^{2g+1} + a x^{g+1} + \alpha^g x$ be a genus $g$ hyperelliptic curve defined over the finite field $\mathbb{F}_q$. Let $\omega$ be a hyperelliptic involution. Then the curve has a non-hyperelliptic involution $\sigma: (x, y) \mapsto (\frac{\alpha}{x},  y\frac{\alpha^\frac{g+1}{2}}{x^{g+1}})$ and equations for quotients of the curve $X$ modulo involutions $\sigma$ and $\omega \sigma$ are following.
\begin{enumerate}
	\item If $g$ is odd then
	\begin{equation}
		\label{eq:X_odd_sigma}
		X/\left< \sigma \right>: y^2 = D_{g}(x,\alpha) + a
	\end{equation}
	 and
	\begin{equation}
		\label{eq:X_odd_omega_sigma}
		X/\left< \omega \sigma \right>: y^2 = (x^2 - 4\alpha)(D_{g}(x,\alpha) + a),
	\end{equation}
	where $D_g(x, \alpha)$ is a Dickson polynomial of degree $g$.
	\item If $g$ is even then
	\begin{equation}
		\label{eq:X_even_sigma}
		X/\left< \sigma\right>: y^2 = (x + 2\sqrt{\alpha}) (D_g(x, \alpha) + a)
	\end{equation} and
	\begin{equation}
		\label{eq:X_even_omega_sigma}
		X/\left< \omega \sigma \right>: y^2 = (x - 2 \sqrt{\alpha}) (D_g(x, \alpha) + a).
	\end{equation}
\end{enumerate}
\end{theorem}
\begin{proof}
The case of equation \eqref{eq:X_odd_sigma} is proved in \cite[\S7.3]{Smith2005}. The case of $b=1$ for equations \eqref{eq:X_odd_sigma} and \eqref{eq:X_even_sigma} is proved in \cite{TauTopVer1991}, but without using Dickson polynomials. We use a similar approach to prove the remaining cases.

Dickson polynomials have a property $D_{g}\left(x+\frac{\alpha}{x}\right) = x^g + (\frac{\alpha}{x})^g$. This allows us to write the equation of $X$ in a form
\[
y^2 = x^{g+1} \left(D_{g}\left(x+\frac{\alpha}{x}, \alpha\right) + a\right).
\]

The function field of the curve $X/\left<\omega \sigma \right>$ is a field of $(\omega \sigma)$-invariant functions on $X$. It is generated by the functions $\xi = x + \frac{\alpha}{x}$ and $\eta = \frac{y}{x^\frac{g-1}{2}} \left( 1 - \frac{\alpha}{x^2} \right)$ in the case of odd $g$ and by $\xi = x + \frac{\alpha}{x}$ and $\eta = \frac{y}{x^\frac{g}{2}} \left( 1 - \frac{\sqrt{\alpha}}{x} \right)$ when $g$ is odd.

By using the property of Dickson polynomials, we find
\[
\eta^2 = (D_2(\xi,b) - 2\alpha) (D_g(\xi, \alpha) + a) = (\xi^2 - 4\alpha) (D_g(\xi, \alpha) + a)
\]
for odd $g$ and
\[
\eta^2 = (\xi - 2\sqrt{\alpha}) (D_g(\xi, \alpha) + a)
\]
when $g$ is even. This proves \eqref{eq:X_odd_omega_sigma} and \eqref{eq:X_even_omega_sigma}.

The case \eqref{eq:X_even_sigma} can be proved in the same way by taking $\xi = x + \frac{\alpha}{x}$ and $\eta = \frac{y}{x^\frac{g}{2}} \left( 1 + \frac{\sqrt{\alpha}}{x} \right)$.
\end{proof}

In the case of \eqref{eq:X_even_sigma} and \eqref{eq:X_even_omega_sigma} the automorphism $\sigma$ is defined over $\mathbb{F}_q[\sqrt{\alpha}]$, so the quotient maps and quotient curves are defined over this field. A map $(x,y) \mapsto (-x, i y)$ is an isomorphism of the curves. Therefore, the curves are twists of degree $2$ to each other.

Since the curve $X$ has the non-hyperelliptic involution $\sigma$, we can split the Jacobian by using the method of Kani-Rosen \cite{KaniRosen1989,Paulhus2007}:
\[
J_X \sim J_{X/\left< \sigma \right>} \times J_{X/\left< \omega  \sigma \right>}.
\]

The equations for quotients are known from the Theorem \ref{th:X_quotients}. Therefore, the problem of point counting on the curve $X$ is reduced to counting points on the quotients. 

The curve $C$ is a curve $X$ with $\alpha = \sqrt[g]{b}$. Automorphisms are defined over the field $\mathbb{F}_q[\sqrt[g]{b}]$ or $\mathbb{F}_q[\sqrt[2g]{b}]$, therefore we have decomposition
\[
J_C(\mathbb{F}_q[\sqrt[g]{b}]) \sim J_{X/\left<\sigma\right>}(\mathbb{F}_q[\sqrt[g]{b}]) \times J_{X/\left<\sigma \omega\right>}(\mathbb{F}_q[\sqrt[g]{b}])
\]
for odd $g$ and
\[
J_C(\mathbb{F}_q[\sqrt[2g]{b}]) \sim J_{X/\left<\sigma\right>}(\mathbb{F}_q[\sqrt[2g]{b}]) \times J_{X/\left<\sigma \omega \right>}(\mathbb{F}_q[\sqrt[2g]{b}])
\]
in case of even $g$.

In addition, for odd genus we have a map $C \to E$ given by $(x,y) \mapsto (x^g, y x^{\frac{g-1}{2}})$, where $E$ is an elliptic curve with equation
\[
y^2 = x^3 + a x^2 + b x.
\]
This map is in fact a quotient map by automorphism $r^2$ for $r$ from Proposition \ref{th:gen_aut_group_of_c_rho}. So, the Jacobian $J_C$ is always split in this case and we have $J_C \sim E \times A$, where $A$ is an abelian variety. Therefore, counting points on the curves of odd genus is reduced to counting points on the abelian variety $A$.

In case of $g = 2$ the decomposition is similar to the decomposition by Satoh \cite{Satoh2009}, but it obtained with different method.

It remains to compute $\#J_C(\mathbb{F}_q)$ from $\#J_C(\mathbb{F}_q[\sqrt[g]{b}])$ for odd $g$ and from $\#J_C(\mathbb{F}_q[\sqrt[2 g]{b}])$ in the even case.

We present a method to compute $\#J_C(\mathbb{F}_q)$ in Section \ref{sec:chi_C_computation}.

\section{Cartier-Manin matrix of the curve $C$}
\label{sec:CM_matrices}
Let $p>2$ and $X$ be a genus $g$ hyperelliptic curve defined by equation
\[
y^2 = f(x),
\]
where $f$ is monic and $\deg{f} = 2g+1$. Denote by $c_i$ coefficients in expansion of $f(x)^{\frac{p-1}{2}}$. A matrix 
\[
W = (w_{i,j}) = (c_{i p - j})
\]
is called Cartier-Manin matrix of the curve $X$.
Let \[W_p = (W^t) \cdot (W^t)^{(p)} \cdot ... \cdot (W^t)^{(p^{n-1})},\] where $(W^t)^{(p^i)}$ denotes a matrix obtained from $W^t$ by raising each of its elements to $p^i$-th power.
Then we have a formula \cite{Manin1961,Yui1978}:
\begin{equation}
\label{eq:Manin_formula}
\chi_{X,q}(T) \equiv T^g |TI - W_p| \pmod{p}.
\end{equation}

In this section we show how to compute this matrices for finite fields of big characteristic, i.e. for finite fields with $p > 2^{160}$.

It is known that the number of points on certain elliptic curves can be expressed through Legendre polynomials. Therefore, some instances of the polynomials from \cite[Table 1,2]{Novoselov2017} can be computed for finite fields of big characteristic using the Schoof-Elkies-Atkin algorithm (see \cite{Schoof1995} and \cite[\S17.2.2]{CohenFrey+2005}). We collect such cases in the following theorem.

\begin{theorem}[\cite{Sun2013,Sun2013b,Sun2013c}]
	\label{th:LP_EC_congruences}
	Let $c \in \mathbb{F}_p, p > 3$. Then
	\begin{enumerate}
		\item $P_{\frac{p-1}{2}}(c) \equiv (\frac{-6}{p}) t_2 \pmod{p}$, where $t_2$ is a trace of Frobenius of an elliptic curve:
		\[
		E_2: y^2 = x^3 - 3(c^2 + 3) x + 2 c (c^2 - 9).
		\]
		\item $P_{\lfloor\frac{p}{3}\rfloor}(c) \equiv (\frac{p}{3}) t_3 \pmod{p}$, where $t_3$ is a trace of Frobenius of an elliptic curve
		\[
		E_3: y^2 = x^3 + 3 (4c-5) x + 2(2c^2 - 14 c + 11).
		\]
		\item $P_{\lfloor\frac{p}{4}\rfloor}(c) \equiv (\frac{6}{p}) t_4 \pmod{p}$, where $t_4$ is a trace of Frobenius of an elliptic curve
		\[
		E_4: y^2 = x^3 - \frac{3}{2}(3 c + 5) x + 9 c + 7.
		\]
		\item $P_{\lfloor\frac{p}{6}\rfloor}(c) \equiv (\frac{3}{p}) t_6 \pmod{p}$, where $p>5$ and $t_6$ is a trace of Frobenius of an elliptic curve
		\[
		E_6: y^2 = x^3 - 3x + 2c.
		\]
	\end{enumerate}
\end{theorem}

Using this theorem we can compute Cartier-Manin matrix of curve $C'$ of genus $3$ completely. For the case of $g > 3$ we get partial information. For example, the polynomial $P_{\frac{p-1}{2}}$ appears in formulae for $g = 5, 7$ in \cite[Table 1,2]{Novoselov2017}. Also, we can obtain new congruences for Legendre polynomials using tables from \cite{Novoselov2017} and decomposition of $J_{C'}$ from Section \ref{sec:jacobian_decomposition}.

To find a Cartier-Manin matrix of the curve $C$, we connect this matrix with the matrix of the curve $C'$ using a theorem.

\begin{theorem}
	\label{th:W_for_C_from_Ct}
	Let $C: y^2 = x^{2g+1} + a x^{g+1} + bx$ be a genus $g$ hyperelliptic curve defined over a finite field $\mathbb{F}_q, q = p^n, p>2, \gcd(p,g)=1$ and let $W = (w_{i,j}(a,b))$ be a Cartier-Manin matrix of $C$, and $m = \frac{p-1}{2}$. Then
	\begin{enumerate}
		\item
		\label{th:W_for_C_from_Ct_1}
		$w_{i,j}(a,b) = b^{m + \frac{m-(ip-j)}{2g}} w_{i,j}\left(\frac{a}{\sqrt{b}}, 1\right) = b^{m + \frac{m-(ip-j)}{2g}} P_{\frac{ip-j}{g} - \frac{p-1}{2g}}(-\frac{a}{2\sqrt{b}})$ for $ip-j \equiv m \pmod{g}$ and
		\item $w_{i,j} = 0$, otherwise.
	\end{enumerate}
\end{theorem}
\begin{proof}
	\begin{align*}
	w_{i,j}(a,b)
	&=
	[x^{ip-j}](x^{2g+1} + a x^{g+1} + bx)^{\frac{p-1}{2}}
	=\\&=
	[x^{ip-j}] \sum\limits_{k_0+k_1+k_2=m}\binom{m}{k_0,k_1,k_2} x^{(2g+1)k_0 + (g+1)k_1 + k_2} a^{k_1} b^{k_2}
	=\\&=
	\sum\binom{m}{k_0,k_1,k_2} a^{k_1} b^{k_2},
	\end{align*}
	where the sum goes all integers $k_0,k_1,k_2 \geq 0$ such that
	\[
	\begin{cases}
	k_0+k_1+k_2=m, \\
	(2g+1)k_0 + (g+1)k_1 + k_2=ip-j.
	\end{cases}
	\]
	Therefore,
	\begin{align*}
	w_{i,j}(a,b)
	&=
	\sum\limits_{k_0 \geq 0}\binom{m}{k_0, \frac{ip-j-m}{g} - 2 k_0, m + \frac{m-(ip-j)}{g} + k_0}
	a^{\frac{ip-j}{g} - 2 k_0} b^{m+\frac{m-(i p - j)}{g} + k_0}
	=\\&=
	b^{m + \frac{m-(ip-j)}{2g}}\sum\limits_{k_0 \geq 0}\binom{m}{k_0, \frac{ip-j-m}{g} - 2 k_0, m + \frac{m-(ip-j)}{g} + k_0}
	\left(\frac{a}{\sqrt{b}}\right)^{\frac{ip-j}{g} - 2 k_0}
	=\\&=
	b^{m + \frac{m-(ip-j)}{2g}} w_{i,j}\left(\frac{a}{\sqrt{b}}, 1\right).
	\end{align*}
	
	By Theorem 3 from \cite{Novoselov2017} we have
	$w_{i,j}(a,b) = b^{m + \frac{m-(ip-j)}{2g}} P_{\frac{ip-j}{g} - \frac{p-1}{2g}}$ for $ip-j \equiv \frac{p-1}{2} \pmod{g}.$
\end{proof}
This theorem allows us to compute Cartier-Manin matrix for the curve $C$ from the matrix of the curve $C'$. Note that we should work in the field $\mathbb{F}_{q^2}$ in the case of $\sqrt{b} \not\in \mathbb{F}_q$, but the result belongs to $\mathbb{F}_{q}$.

The curves $C$ and $C'$ are isomorphic over the field $\mathbb{F}_{q}[\sqrt[4g]{b}]$, but the theorem works even in the case when the curves are not isomorphic over $\mathbb{F}_q$.

Now, we can derive all possibilities for $\chi_C(T) \pmod{p}$ using the same method as in \cite{Novoselov2017}. First, it follows from Theorem \ref{th:W_for_C_from_Ct} that Cartier-Manin matrix $W$ of $C$ is (generalized) permutation matrix with permutation $\sigma$ defined by
\[
\sigma(i) \equiv ip - \frac{p-1}{2} \pmod{g}.
\]
A decomposition of $\sigma = \sigma_1 \cdot ... \cdot \sigma_k$ into disjoint cycles corresponds to factorization of the characteristic polynomial of the matrix $W$:
\begin{equation}
\label{eq:chi_factorization_from_permutation}
  \chi_W(T)
=
\prod_{i=1}^{k}\left( T^{|\sigma_i|} - \prod_{j=1}^{|\sigma_i|} w_{\sigma_i(j), \sigma_i(j+1)} \right).
\end{equation}

The same result holds for the matrix $W_p$ if we take $\sigma \equiv i p^n - \frac{p^n-1}{2} \pmod{p}$.

By formula \eqref{eq:Manin_formula}, we have
\begin{equation}
\label{eq:Manin_formula_2}
\chi_{C,q}(T) \equiv T^g \chi_{W_p}(T) \pmod{p}.
\end{equation}

By combining \eqref{eq:Manin_formula_2} with Theorem \ref{th:W_for_C_from_Ct} we can now express $\chi_{C,q}(T) \mod{p}$ in terms of Legendre polynomials.

Fixing genus $g$ and a number $s$, such that $p \equiv s \pmod{g}$  for odd genus $g$ and $p \equiv s \pmod{2g}$ in case of $g$ is even, we fix permutation $\sigma$. So, for each genus we have $g$ variants for $\chi_C(T) \pmod{p}$. We enumerate them all in Table \ref{table:C_frobs_mod_p} for $g = 1-7$. The polynomials are given in factored form (over extension of $\mathbb{F}_q$), but all coefficients belong to $\mathbb{F}_p$ after expansion.

\textit{Congruences for Legendre polynomials.} From Section \ref{sec:jacobian_decomposition}, we have $J_{C'}(\mathbb{F}_q) \sim J_{X'_1} \times J_{X'_2}$. Therefore, \[
\chi_{C',q}(T) = \chi_{X'_1, q}(T) \chi_{X'_2, q}(T).
\]
Combining this with \eqref{eq:Manin_formula_2} we obtain
\[
T^g \chi_{W_p}(T) \equiv \chi_{X'_1, q}(T) \chi_{X'_2, q}(T) \pmod{p}.
\]
A number of congruences for Legendre polynomials $P_{\frac{ip-j}{g} - \frac{p-1}{2g}}$ can be found by comparing coefficients of polynomials from two sides of this equation. We give an example of such congruences in Section \ref{sec:pc_genus4}.

The Jacobians $J_{X'_1}$ and $J_{X'_2}$ are "generically" absolutely simple \cite[Cor.6]{TauTopVer1991}. So, we have obtained a connection of characteristic polynomials of absolutely simple abelian varieties with Legendre polynomials.

Previous results \cite{BrillhartMorton2004,Sun2013,Sun2013b,Sun2013c} connect Legendre polynomials with elliptic curves. But transition to hyperelliptic curves gives us more results.

\section{Computing the order of $J_C(\mathbb{F}_q)$.}
\label{sec:chi_C_computation}
Since we know the decomposition of Jacobian over the extension of finite field, it remains to compute $\chi_{C,q}(T)$ from $\chi_{C,q^k}(T)$, where $k$ such that $\mathbb{F}_q[\sqrt[g]{b}] \simeq \mathbb{F}_{q^k}$ for odd $g$ and $\mathbb{F}_q[\sqrt[2g]{b}] \simeq \mathbb{F}_{q^k}$ for even $g$.

We do this by using a formula for $L$-polynomials from \cite[p.195]{Stichtenoth2009}:
\begin{equation}
\label{eq:L_from_Lk}
L_{C,q^k}(T^k) = \prod_{\zeta^k=1} L_{C,q}(\zeta T).
\end{equation}

By comparing coefficients, we obtain a system of $2gk$ equations of $g$ unknowns $a_1, ..., a_g$.

This system is big in general, so to optimize process we adopt a step-by-step method for genus $2$ case from \cite{GuillevicVergnaud2012}. This gives us an Algorithm \ref{alg:general}.

\begin{algorithm}[H]
	\SetAlgoLined
	\caption{Computation of $\chi_{C,q}(T)$ for hyperelliptic curve $C: y^2 = x^{2g+1} + a x^{g+1} + b x$.
		\label{alg:general}
	}
	\KwIn{$a,b \in \mathbb{F}_q$, $p > 2$.}
	\KwOut{Coefficients $(a_{1}, ..., a_{g})$ of $\chi_{C,q}(T)$.}
	\nl Let $k$ be such that $\mathbb{F}_q[\sqrt[g]{b}] \simeq \mathbb{F}_{q^k}$ if $g$ is odd and $\mathbb{F}_q[\sqrt[2 g]{b}] \simeq \mathbb{F}_{q^k}$ if $g$ is even\;
	\nl Let $k = k_1 \cdot ... \cdot k_m$ be a prime factorization of $k$ with $k_1 \leq ... \leq k_m$\;
	\nl $n \leftarrow k$\;
	\nl \label{alg:gen:step_computation_of_X_i}
	Compute $L_{X_1,q^n}(T)$ and $L_{X_2,q^n}(T)$ for $X_1 := X/\left<\sigma\right>$ and $X_2 := X/\left<\sigma \omega\right>$ from Theorem \ref{th:X_quotients}\;
	\nl $L_{C,q^n}(T) \leftarrow L_{X_1,q^n}(T) L_{X_2,q^n}(T)$\;
	\nl $list \leftarrow \{ (a_{1,n}, ..., a_{g,n}) \}$\;
	\nl \For{$j$ from $1$ to $m$}{
		\nl $i \leftarrow \frac{n}{k_j}$\;
		\nl $S \leftarrow \{\}$\;
		\nl \ForEach{$(a_{1,n}, ..., a_{g,n}) \in list$}{
			\nl \label{alg:gen:step_seq} Obtain a system of equations for coefficients $a_{1,i}, ..., a_{g,i}$ of $L_{C,q^i}(T)$ from the formula $L_{C,q^n}(T^{k_j}) = \prod_{\zeta^{k_j}=1}L_{C,q^i}(\zeta T)$\;
			\nl Solve the system using inequalities $|a_{m,i}| \leq \binom{2g}{m} q^{\frac{i m}{2}}$ for $m = 1,...,g$ to obtain a list $S'$ of possible tuples $(a_{1,i}, ..., a_{g,i})$\;
			\nl Exclude extra solutions from $S'$ by multiplying random points of $J_C(\mathbb{F}_{q^i})$ by $1 + a_{1,i} + ... a_{g,i} + a_{1,i} q^i + ... + a_{g,i} q^{g i}$\;
			\nl $S \leftarrow S \cup S'$\;
		}
	\nl $list \leftarrow S$\;
	\nl $n \leftarrow i$\;
	}
	\Return $list[1]$\;
\end{algorithm}

The system of equations in the step \ref{alg:gen:step_seq} can be precomputed for fixed genus. It contains at most $2gk_j$ equations with $g$ unknowns of degree $k_j$.

To solve the system, we use resultants to eliminate variables $a_{2,i}, ..., a_{g,i}$. This gives us one polynomial in variable $a_{1,i}$. To find roots of this polynomial, we factor it over $\mathbb{F}_l$ (for prime $l > 4g\sqrt{q}$) in the same way as it is done in \cite{Satoh2009} (for $g=2$). For each solution $a_{1,i}$ we substitute it in previous equation which depends only on $a_{1,i}$ and some other variable $a_{m,i}$. Then we factor resulting polynomial to find possible solutions for $a_{m,i}$ and so on. In the end, we obtain a list of possible tuples $(a_{1,i}, ..., a_{g,i})$. To exclude extra solutions, we first use Hasse-Weil bound and after that we eliminate remaining solutions by choosing random points on $J_C(\mathbb{F}_{q^i})$ and multiplying candidates for $\#J_C(\mathbb{F}_{q^i})$ by it.

For curves with simple Jacobians this gives us an unique solution in the end of algorithm. For curves with non-simple Jacobian the algorithm returns result up to twists of factors in the Jacobian decomposition.

\textit{Twists.} To compute $L_{X_1,q^n}(T)$ and $L_{X_2,q^n}(T)$ in step \ref{alg:gen:step_computation_of_X_i}, we first compute characteristic polynomials for the twists of $X_1, X_2$, defined over $\mathbb{F}_q[\sqrt{b}]$. After that, we determine $L_{X_1,q^n}(T)$ and $L_{X_2,q^n}(T)$ from characteristic polynomials of twists using precomputed systems of equations for the step \ref{alg:gen:step_seq}.

The twists for even $g$ are given by equations
\[
\tilde{X}_{1,2}: y^2 = (x \pm 2) (D_g(x,1) + a/\sqrt{b}).
\]
For odd $g$, we have
\[
\tilde{X}_1: y^2 = D_g(x, 1) + a / \sqrt{b}
\]
and
\[
\tilde{X}_{2}: y^2 = (x^2 - 4) (D_g(x, 1) + a / \sqrt{b}).
\]

\begin{remark}
The curve $\tilde{X}_1$ is a curve with explicit real multiplication \cite[\S7.1.1]{Abelard2018}. Due to recent result of Abelard \cite{Abelard2018b} the problem of counting points on this curve has complexity $O(\log^8{q})$  for any fixed genus $g$.
\end{remark}

We provide implementation of the Algorithm \ref{alg:general} for the case of $g = 4$ in Section \ref{sec:pc_genus4}. The algorithm for $g = 3$ in Section \ref{sec:pc_genus3} uses Cartier-Manin matrices instead.

\section{Genus 3}
\label{sec:pc_genus3}
In this case, we have $J_C(\mathbb{F}_q) \sim E \times A$ for some abelian variety $A$ of dimension $2$ and to determine $\chi_{C,p}(T)$ we have to find a characteristic polynomial $\chi_{A,p}(T)$ of $A$. Let \[
\chi_{A,p}(T) = T^4 - b_1 T^3 + b_2 T^2 - b_1 p T + p^2.\]
To find coefficients $b_1, b_2$, we compute $\chi_{A,p}(T) \pmod{p}$ by using results from Section \ref{sec:CM_matrices}. After that, we determine coefficients $b_1, b_2$ by using inequalities $|b_1| \leq 4 \sqrt{p}$ and $|b_2| \leq 6 p$ in the same way as in \cite[Alg.1]{FurKawTak2003}. To do this, we make a list of candidates for $b_1,b_2$ and determine the right one by multiplying random point in $J_C(\mathbb{F}_p)$ by candidate for $\#J_C(\mathbb{F}_p) = 1 + p^2 + b_1 (p + 1) + b_2$. 

The computation of $\chi_{C,p}(T)$ is as follows. From Table \ref{table:C_frobs_mod_p} we have
\[
\chi_{C,p}(T) \equiv T^3 (T - b_2 P_{\frac{p - 1}{2}}) (T - b_6^5 P_{\frac{p - 1}{6}}) (T - b_6 P_{\frac{p - 1}{6}}) \pmod{p}
\]
for $p \equiv 1 \pmod{3}$ and
\[
\chi_{C,p} \equiv T^3 (T - b_2 P_{\frac{p - 1}{2}}) (T^2 - b_2^2 P_{\frac{p - 5}{6}}^2) \pmod{p}
\]
for $p \equiv 2 \pmod{3}$.
Computation of Legendre polynomials $P_\frac{p-1}{2},P_\frac{p-1}{6},P_\frac{p-5}{6}$ in case of $\sqrt{b} \in \mathbb{F}_p$ can be done by Theorem \ref{th:LP_EC_congruences} using SEA-algorithm.

In case of $b \not\in \mathbb{F}_p$ values of our Legendre polynomials are not in $\mathbb{F}_p$. So in this case, we compute $\chi_{C,p^2}(T) \pmod{p}$ and restore $\chi_{C,p}(T) \pmod{p}$ from this polynomial by solving system
\begin{equation}
\label{eq:g2_L_from_L2_mod_p}
b_{1,2} \equiv b_1^2 - 2 b_2,~
b_{2,2} \equiv b_2^2 \pmod{p}.
\end{equation}

By using method described in Section \ref{sec:CM_matrices}, we obtain
\[
\chi_{C,p^2}(T) \equiv
T^3
(T - \sqrt{b}^{\frac{p^2-1}{2}} P_{\frac{p - 1}{2}}^{p + 1})
(T - \sqrt{b}^{\frac{5 p^2 - 5}{6}} P_{\frac{p - 1}{6}}^{p + 1})
(T - \sqrt{b}^{\frac{p^2 - 1}{6}} P_{\frac{p - 1}{6}}^{p + 1})
\pmod{p}
\]
for $p \equiv 1 \pmod{3}$
and
\[
\chi_{C,p^2}(T) \equiv
T^3
(T - \sqrt{b}^{\frac{p^2 - 1}{2}} P_{\frac{p - 1}{2}}^{p + 1})
(T - \sqrt{b}^{\frac{5 p^2 - 5}{6}} P_{\frac{p - 5}{6}}^{p + 1})
(T - \sqrt{b}^{\frac{p^2 - 1}{6}} P_{\frac{p - 5}{6}}^{p + 1})
\pmod{p}.
\]
in case of $p \equiv 2 \pmod{3}$.

The computation of $P_{\frac{p - 1}{2}}^{p + 1}, P_{\frac{p - 1}{6}}^{p + 1}, P_{\frac{p - 5}{6}}^{p + 1}$ is equivalent to computing of Frobenius traces $t_{2,2}, t_{6,2}$ for elliptic curves $E_2, E_6$ over $\mathbb{F}_{p^2}$. It can be done using SEA-algorithm.
After solving \eqref{eq:g2_L_from_L2_mod_p} the determination of $\chi_{C,p}(T)$ from $\chi_{C,p}(T) \pmod{p}$ is the same as for the case $\sqrt{b} \in \mathbb{F}_p$. However, in this case we have more possible candidates.

The described method leads to Algorithm \ref{alg:genus_3}.

\begin{algorithm}[H]
	\SetAlgoLined
	\caption{Computation of $\chi_{C,p}(T)$ for hyperelliptic curve $C: y^2 = x^7 + a x^4 + b x$.
		\label{alg:genus_3}
	}
	\KwIn{$a,b \in \mathbb{F}_p$, $p > 2$.}
	\KwOut{$\chi_{C,p}(T)$.}
	
	\nl Compute a trace of Frobenius $t_2$ for elliptic curve $y^2 = x^3 + a x^2 + b x$ over $\mathbb{F}_p$\;

	\nl \eIf{$\sqrt{b} \in \mathbb{F}_p$}{
		\nl $c \leftarrow -\frac{a}{2\sqrt{b}} \in \mathbb{F}_p$ \;
		\nl Compute a trace of Frobenius $t_6$ for elliptic curve $E_6: y^2 = x^3 - 3x + 2c$ over $\mathbb{F}_p$\;
		\nl \uIf{$p \equiv 1 \pmod{3}$}{
			\nl	$\tilde{b}_1 \leftarrow -(\frac{3}{p}) t_6 (\sqrt{b}^{\frac{5p-5}{6}} +\sqrt{b}^{\frac{p-1}{6}})$\;
			\nl $\tilde{b}_2 \leftarrow t_6^2$\;
		}{
			\nl \uElseIf{ $p \equiv 2 \pmod{3}$ }{
				\nl $\tilde{b}_1 \leftarrow 0$\;
				\nl $\tilde{b}_2 \leftarrow - t_6^2$\;
			}
		}
		\nl $\chi_{A,p}(T) \pmod{p} \leftarrow T^4 + \tilde{b}_1 T^3 + \tilde{b}_2 T^2 + \tilde{b}_1 p T + p^2$\;
		\nl Compute $\chi_{A,p}(T)$ from $\chi_{A,p}(T) \pmod{p}$\;
		\Return{$(T^2 - t_6 T + p) \chi_{A,p}(T)$.}
	}{
		\nl $c \leftarrow -\frac{a}{2\sqrt{b}} \in \mathbb{F}_{p^2}$ \;
		\nl Compute a trace of Frobenius $t_{6,2}$ for elliptic curve $E_6: y^2 = x^3 - 3x + 2c$ over $\mathbb{F}_{p^2}$\;
		\nl $\tilde{b}_{1,2} \leftarrow 
			\pm t_{6,2}(\sqrt{b}^{\frac{5p^2-5}{6}}+\sqrt{b}^{\frac{p^2-1}{6}})$ \;
		\nl $\tilde{b}_{2,2} \leftarrow t_{6,2}^2$\;
		\nl $\chi_{A,p^2}(T) \pmod{p} \leftarrow T^4 + \tilde{b}_{1 ,2}T^3 + \tilde{b}_{2,2} T^2 + \tilde{b}_{1,2} p^2 T + p^4$\;
		\nl Compute $\chi_{A,p}(T)$ from $\chi_{A,p^2}(T) \pmod{p}$\;
	}
	\Return $(T^2 - t_2 T + p) \chi_{A,p}(T)$
\end{algorithm}

\textit{Complexity.}
The most time-consuming part of the algorithm is the computation of traces of Frobenius $t_2, t_6, t_{6,2}$. It can be done with Schoof-Elkies-Atkin algorithm. Therefore, the complexity of Algorithm \ref{alg:genus_3} is $\tilde{O}(\log^4{p})$.

We implemented this algorithm in Sage \cite{sagemath}.
In the case of split abelian variety $A$ there can be many candidates for coefficients $(b_1, b_2)$ which pass the test by multiplying by random point of $J_C$. So in this case, we return the list of polynomials. This situation occurs because solutions of system of equations correspond to twists of elliptic curves in decomposition of $J_C$ and therefore polynomials are expected to have a common factor.

\begin{example}
	Let $p = fa16da0d09e774b881f9a8836ccc55d1$ (128-bit),
	
	$a = e565b9386557e274880cd235cd733d8c$,
	
	$b = aacc117a8fefc11ca37befa58beb2be9$.
	
	By applying the algorithm, we obtain
	
	$t_2 = 4d089c83177cc1f8$,
	
	$b_1 = 945309b8f7ad6614$,
	
	$b_2 = a426bbdfdd37d53206f17355b5106441$.
	
	The computation took $7.01$ sec. on laptop with Core i7-4700HQ CPU  clocked at 2.40GHz.
	
	We have $\mathbb{Q}$-irreducible polynomials $\chi_{p,A}(T)$ and $\chi_{p^2,A}(T)$. Therefore, $A$ is simple over $\mathbb{F}_p$ and $\mathbb{F}_{p^2}$.
	The number of points on $A$ is
	\begin{equation*}
	\#A(\mathbb{F}_p) = f450a3ebab4ff949332678949cc73c566b8ea584ae41c426300701830bde70c9.
	\end{equation*}
	It has a large divisor $r = e27313cdfeb582ed1111bb6c69c2ac8686d6674146864b7$ of cryptographic size $187$ bit. Thus, abelian variety $A$ is suitable for cryptography based on discrete logarithm problem.
\end{example}

Another way for point-counting in genus $3$ case is to derive an explicit formulae for the $\#J_C(\mathbb{F}_q)$ as it is done in \cite{GuillevicVergnaud2012} for genus $2$. A complete list of the characteristic polynomials for the genus $3$ curves to appear in \cite{NovoselovBoltnev2019}.

\section{Genus 4}
\label{sec:pc_genus4}
In this case from the results of Section \ref{sec:jacobian_decomposition}, we have \[
J_C(\mathbb{F}_q[\sqrt[8]{b}]) \sim J_{X_1}(\mathbb{F}_q[\sqrt[8]{b}]) \times J_{X_2}(\mathbb{F}_q[\sqrt[8]{b}]),
\] where
\[
X_1: y^2 = (x + 2 \sqrt[8]{b}) (x^4 - 4 x^2 \sqrt[4]{b} + 2 \sqrt{b} +  a)
\]
and 
\[
X_2: y^2 = (x - 2 \sqrt[8]{b}) (x^4 - 4 x^2 \sqrt[4]{b} + 2 \sqrt{b} +  a).
\]

To compute the order of Jacobian $\#J_C(\mathbb{F}_q)$, we compute characteristic polynomials of the genus $2$ curves $X_1, X_2$ over $\mathbb{F}_q[\sqrt[8]{b}]$ and determine $\chi_{C,q}(T)$ from them. Since the curves $X_1, X_2$ are isomorphic over $\mathbb{F}_q[\sqrt[8]{b}, \sqrt{-1}]$, it's enough to compute one of the characteristic polynomials and choose the sign in coefficients of second one depending on whether $-1$ is a square in $\mathbb{F}_q$ or not.
For fast calculation of $\chi_{X_1}(T)$ over $\mathbb{F}_q[\sqrt[8]{b}]$, we use twist of $X_1$ defined over $\mathbb{F}_q[\sqrt{b}]$ by equation
\[
\tilde{X}_1:  y^2 = (x + 2) (x^4 - 4 x^2 + 2 +  a / \sqrt{b}).
\]
We compute the characteristic polynomial of $\tilde{X}_1$ over $\mathbb{F}_q[\sqrt{b}]$ and determine the characteristic polynomial of $X_1$ over $\mathbb{F}_q[\sqrt[8]{b}]$ by using formulae from \cite[p.5]{GuillevicVergnaud2012}:
\[
a_{1,2} = 2 a_{2} - a_{1}^2,
\]
\[
a_{2,2} = a_{2}^2 - 4 q \cdot a_{2} + 2 q^2 + 2 q \cdot a_{1,2}.
\]

After computation of characteristic polynomial of $X_1$ over $\mathbb{F}_q[\sqrt[8]{b}]$ we compute step-by-step the characteristic polynomial $\chi_{C,q}(T)$ descending over quadratic extensions. We do this by solving following system of equations obtained from \eqref{eq:L_from_Lk}:
\begin{equation}
\label{eq:g4_a12}
a_{1,2} = -a_{1}^2 + 2 a_{2},
\end{equation}
\begin{equation}
\label{eq:g4_a22}
a_{2,2} = a_{2}^2 - 2 a_{1} a_{3} + 2 a_{4},
\end{equation}
\begin{equation}
\label{eq:g4_a32}
a_{3,2} = -2 q \cdot a_{1} a_{3} + 2 q^2 a_{2} - a_{3}^2 + 2 a_{2} a_{4},
\end{equation}
\begin{equation}
\label{eq:g4_a42}
a_{4,2} = 2 q^3 \cdot a_{1,2} - 4 q \cdot a_{3}^2 + 4 q \cdot a_{2} a_{4} - 4 q^2 \cdot a_{4} + a_{4}^2 +
2 q^2 \cdot a_{2,2} - 2 q \cdot a_{3,2} + 2 q^4.
\end{equation}

From \eqref{eq:g4_a12} and \eqref{eq:g4_a22} we have for $a_1 \neq 0$:
\begin{equation}
\label{eq:g4_a2}
a_{2} = (a_{1,2} + a_{1}^2) / 2,
\end{equation}
\begin{equation}
\label{eq:g4_a3}
a_{3} = (a_{2}^2 + 2 a_{4} - a_{2,2}) / (2 a_{1}).
\end{equation}

By substituting \eqref{eq:g4_a3} to \eqref{eq:g4_a32} and \eqref{eq:g4_a42}, we obtain
\begin{equation}
\label{eq:g4_a4}
\begin{split}
\frac{a_{4}^2}{a_{1}^2}
+
\left(\frac{a_{2}^2 - a_{2,2}}{a_{1}^2}-2 a_{2}+2 q\right) a_{4}
-
a_{2,2} q + a_{2}^2 q+a_{3,2} + \\ + \frac{a_{2,2}^2-2 a_{2}^2 a_{2,2} + a_{2}^4}{4 a_{1}^2} - 2 a_{2} q^2 = 0
\end{split}
\end{equation}
and
\begin{equation}
\label{eq:g4_a4_2}
\begin{split}
\left(1-\frac{4 q}{a_{1}^2}\right) a_{4}^2 
+
4 q \left(\frac{a_{2,2} - a_{2}^2}{a_{1}^2} +  a_{2} - q\right) a_{4}
+
2 q^4 
+
2 q^3 a_{1,2}
+ \\ +
2 q^2 a_{2,2}
- 2 q a_{3,2}
-\frac{(a_{2,2}^2+ 2 a_2^2 a_{2,2}-a_2^4) q}{a_1^2}
-a_{4,2} = 0.
\end{split}
\end{equation}
After substituting \eqref{eq:g4_a2} in above equations, eliminating $a_4$ by taking resultant, and dividing by $a_1^4$, we obtain degree $16$ polynomial in $a_1$
\begin{equation}
\label{eq:g4_a1_relation}
a_{1}^{16}
+
c_{14} a_{1}^{14} 
+
c_{12} a_{1}^{12}
+
c_{10} a_{1}^{10} 
+
c_8 a_{1}^8
+
c_6 a_{1}^6
+
c_4 a_{1}^4
+
c_2 a_{1}^2
+
c_0 = 0,
\end{equation}
where coefficients $c_i$ are given in Appendix \ref{appendix:a1_relation}.

From this, we have at most $16$ possible values for $a_1$. To find them, we use the same technique as in \cite[\S4]{Satoh2009}. We factor this polynomial over $\mathbb{F}_l$ for $l > 16 \sqrt{q}$ and exclude solutions which does not satisfy the bound $|a_1| \leq 8 \sqrt{q}$.
For each $a_1$ there are at most $2$ possible tuples $(a_2, a_3, a_4)$. So, we have obtained $32$ candidates for $\#J_C(\mathbb{F}_q)$. The right one can be found by taking random points on Jacobian and multiplying it by candidate for the Jacobian order.
The described method leads to Algorithm \ref{alg:genus_4}.

\begin{algorithm}[H]
	\SetAlgoLined
	\caption{Computation of $\chi_{C,q}(T)$ for hyperelliptic curve $C: y^2 = x^9 + a x^5 + b x$.
		\label{alg:genus_4}
	}
	\KwIn{$a,b \in \mathbb{F}_q$.}
	\KwOut{$(a_{1}, a_{2}, a_{3}, a_{4})$ the coefficients of $\chi_{C,q}(T)$.}
	
	\nl Find $k$ such that $\mathbb{F}_q[\sqrt[8]{b}] \simeq \mathbb{F}_{q^k}$\;
	\nl Compute $\chi_{X_{1,q^k}}(T) = T^4 + b_{1,k} T^3 + b_{2,k} T^2 + b_{1,k} q^k T + q^{2k}$\;
	\nl \eIf{$\sqrt{-1} \in \mathbb{F}_{q^k}$}{
		\nl $a_{1,k} \leftarrow 2 b_{1,k}$ and $a_{2,k} \leftarrow 2 b_{2,k}+b_{1,k}^2$\;
		\nl $a_{3,k} \leftarrow 2 b_{1,k}q^k+2 b_{1,k} b_{2,k}$\;
		\nl $a_{4,k} \leftarrow 2 q^{2 k}+2 b_{1,k}^2 q^k+b_{2,k}^2$\;
	}{
		\nl $a_{1,k} \leftarrow 0$ and $a_{3,q^k} \leftarrow 0$\;
		\nl $a_{2,k} \leftarrow 2 b_{2,k}-b_{1,k}^2$\;
		\nl $a_{4,k} \leftarrow 2 q^{2 k} - 2 b_{1,k}^2 q^k + b_{2,k}^2$\;
	}
	\nl $i = k$\;
	\nl $list = \{(a_{1,i}, a_{2,i},a_{3,i},a_{4,i})\}$\;
	\nl \While{$i\neq 1$}{
		\nl $S \leftarrow \{\}$\;
		\nl \ForEach{$(a_{1,{i}}, a_{2,{i}},a_{3,{i}},a_{4,{i}}) \in list$}{
			\nl Factor polynomial \eqref{eq:g4_a1_relation} over $\mathbb{F}_l$ for prime $l > 16 \sqrt{q}$ to obtain a list of possible $a_{1,{i/2}}$\;
			\nl Make a list $S'$ of all possible tuples $(a_{1,{i/2}}, a_{2,{i/2}},a_{3,{i/2}},a_{4,{i/2}})$ using \eqref{eq:g4_a2}, \eqref{eq:g4_a3}, and \eqref{eq:g4_a4} which satisify Hasse-Weil bound\;
			\nl Exclude extra tuples from $S'$ by multiplying the order candidate by the random points on $J_C(\mathbb{F}_{q^{i/2}})$\;
			\nl $S \leftarrow S \cup S'$\;	
		}
		\nl $list \leftarrow S$\;
		\nl	$i \leftarrow i / 2$\;
	}
	\nl \Return{list[1]};
\end{algorithm}

\textit{Complexity.} The factorization of polynomial over finite field $\mathbb{F}_l$ in Algorithm \ref{alg:genus_4} takes time $\tilde{O}(\log{l}) = \tilde{O}(\log{q})$ (see \cite[Th.14.14, p.390]{GathenGerhard2013}). Computation of the characteristic polynomial $\chi_{X_1,p}(T)$ of genus $2$ curve $X_1$ can be done in time $\tilde{O}(\log^8{p})$ using Gaudry-Schost algorithm \cite{GaudrySchost2012}. Therefore, the Algorithm \ref{alg:genus_4} has complexity $\tilde{O}(\log^8{p})$.

Note that the algorithm is less efficient than SEA algorithm with complexity $\tilde{O}(\log^4{p})$, but still more efficient than general algorithms \cite{Pila1990,HuangIerardi1998} for point-counting on genus $4$ curves.

\textit{Implementation of algorithm.} We implemented the algorithm in Sage\cite{sagemath} computer algebra system.

\begin{example}
	Let $p = 4398046511233$, $a = 4231746819984$, $b = 141248343157$.
	Applying the algorithm we obtained coefficients of $\chi_C(T)$:
	\[
	a_1 = - 2112224, ~ a_2 = 2230745113088,~ a_3 = 4306063463022049120,
	\]
	and
	\[
	a_4 = 2745301697312802596344066.
	\]
	The polynomial $\chi_p(T)$ is $\mathbb{Q}$-irreducible and therefore $J_C$ is simple.
	
	The order of Jacobian 
	\[
	\#J_C(\mathbb{F}_p) = fffff7f1c920731feb75b80a59bb590a917dec3284 
	\]
	has size $167$ bit.
	
	The computation took $42$ min. $24$ sec. on laptop with Core i7-4700HQ CPU clocked at 2.40GHz.
\end{example}

\textit{New congruences for Legendre polynomials.}
From decomposition of the curve $C'$ and results for genus $4$ from \cite[Table 1,2]{Novoselov2017} or Table \ref{table:C_frobs_mod_p} we obtain new congruences for Legendre polynomials:
\[
P_{\frac{p-1}{8}}(\rho) \equiv \frac{-b_1 \pm \sqrt{d}}{2} \pmod{p},
\]
\[
P_{\frac{3p-3}{8}}(\rho) \equiv \frac{2 b_2}{-b_1 \pm \sqrt{d}} \pmod{p}.
\]
Where $\rho \in \mathbb{F}_p$, $d = b_1^2 - 4 b_2$, $b_1, b_2$ are coefficients of the Frobenius polynomial of the curve $X'_1$ with $c = -2\rho$.

\begin{example}
It is easy to find an example where the curve $X'_1$ has absolutely simple Jacobian.
If $c = 7$ and $p = 7$ we have $\chi_{D,p}(T) = T^4 - 4 T^3 + 16 T^2 - 28 T + 49$. This polynomial is $\mathbb{Q}$-irreducible, therefore $J_D$ is simple. By result of \cite{ChouKani2014} if a simple ordinary abelian surface is geometrically split, then it splits over an extension of base field of degree at most $6$. By straightforward check, we can see that all $\chi_{D,p^k}(T)$ for $k \leq 6$ are $\mathbb{Q}$-irreducible. Therefore, the curve $D$ with $c=3,p=7$ has absolutely simple Jacobian. Thus, we have obtained congruences for Legendre polynomials $P_\frac{p-1}{8},P_{\frac{3p-3}{8}}$ which connect them with genus $2$ curves that have absolutely simple Jacobians. 
\end{example}

\section*{Conclusion}
In this paper, we have derived two algorithms to compute the number of points on the Jacobian of hyperelliptic curves $C$ in case of $g = 3$ and $g = 4$ with complexity $\tilde{O}(\log^4{p})$ and $\tilde{O}(\log^8{p})$. We obtained a complete list of $\chi_{C,p}(T) \pmod{p}$ (Table \ref{table:C_frobs_mod_p}).

The method for deriving new congruences for Legendre polynomials were obtained in Section \ref{sec:CM_matrices}. And new congruences for Legendre polynomials $P_\frac{p-1}{8}$ and $P_\frac{3p-3}{8}$ were found by using this method in Section \ref{sec:pc_genus4}.

This work is an extended and refined version of preliminary results presented by author on the conference SibeCrypt'18 \cite{Novoselov2018}.

\begin{table}%
\centering\scriptsize
\caption{Characteristic polynomials modulo $p$ for hyperelliptic curves of the form $C: y^2 = x^{2g+1} + a x^{g+1} + b x$ over a prime finite field $\mathbb{F}_p, p>2, p \not|~g, P_m := P_m(-\frac{a}{2\sqrt{b}})$, $b_{i} := \sqrt{b}^{\frac{p-1}{i}}$.}
\label{table:C_frobs_mod_p}
\begin{tabular}{|c|l|l|}
\hline
$g$ &
\rule{0pt}{10pt}
conditions           & $\chi_{C,p}(T) \pmod{p}$ \\ \hline
2&
$p \equiv 1 \pmod{4}$ &
\rule{0pt}{10pt}
$T^2 (T - b_4^3 P_{\frac{p - 1}{4}}) (T - b_4 P_{\frac{p - 1}{4}})$ 
\\\hline
2&
$p \equiv 3 \pmod{4}$ &
\rule{0pt}{10pt}
$T^2 (T^2 - b_2^2 P_{\frac{p - 3}{4}}^2)$ 
\\\hline
3&
$p \equiv 1 \pmod{3}$ &
\rule{0pt}{10pt}
$T^3 (T - b_2 P_{\frac{p - 1}{2}}) (T - b_6^5 P_{\frac{p - 1}{6}}) (T - b_6 P_{\frac{p - 1}{6}})$ 
\\\hline
3&
$p \equiv 2 \pmod{3}$ &
\rule{0pt}{10pt}
$T^3 (T - b_2 P_{\frac{p - 1}{2}}) (T^2 - b_2^2 P_{\frac{p - 5}{6}}^2)$ 
\\\hline
4&
$p \equiv 1 \pmod{8}$ &
\rule{0pt}{10pt}
$T^4 (T - b_8^5 P_{\frac{3 p - 3}{8}}) (T - b_8^3 P_{\frac{3 p - 3}{8}}) (T - b_8^7 P_{\frac{p - 1}{8}}) (T - b_8 P_{\frac{p - 1}{8}})$ 
\\\hline
4&
$p \equiv 3 \pmod{8}$ &
\rule{0pt}{10pt}
$T^4
(T^2 - b_2^3 P_{\frac{3 p - 1}{8}} P_{\frac{p - 3}{8}})
(T^2 - b_2 P_{\frac{3 p - 1}{8}} P_{\frac{p - 3}{8}})$ 
\\\hline
4&
$p \equiv 5 \pmod{8}$ &
\rule{0pt}{10pt}
$T^4
(T^2 - b_4^5 P_{\frac{3 p - 7}{8}} P_{\frac{p - 5}{8}})
(T^2 - b_4^3 P_{\frac{3 p - 7}{8}} P_{\frac{p - 5}{8}})$ 
\\\hline
4&
$p \equiv 7 \pmod{8}$ &
\rule{0pt}{10pt}
$
T^4
(T^2 - b_2^2 P_{\frac{3 p - 5}{8}}^2)
(T^2 - b_2^2 P_{\frac{p - 7}{8}}^2)
$ 
\\\hline
5&
$p \equiv 1 \pmod{5}$ &
\rule{0pt}{10pt}
$T^5
(T - b_2 P_{\frac{p - 1}{2}})
(T - b_{10}^7 P_{\frac{3 p - 3}{10}})
(T - b_{10}^3 P_{\frac{3 p - 3}{10}})
(T - b_{10}^9 P_{\frac{p - 1}{10}})
(T - b_{10} P_{\frac{p - 1}{10}})
$ 
\\\hline
5&
$p \equiv 2 \pmod{5}$ &
\rule{0pt}{10pt}
$
T^5
(T^4 - b_2^4 P_{\frac{3 p - 1}{10}}^2 P_{\frac{p - 7}{10}}^2)
(T - b_2 P_{\frac{p - 1}{2}})
$ 
\\\hline
5&
$p \equiv 3 \pmod{5}$ &
\rule{0pt}{10pt}
$
T^5
(T^4 -b_2^4 P_{\frac{3 p - 9}{10}}^2 P_{\frac{p - 3}{10}}^2)
(T - b_2 P_{\frac{p - 1}{2}})
$ 
\\\hline
5&
$p \equiv 4 \pmod{5}$ &
\rule{0pt}{10pt}
$T^5
(T^2 - b_2^2 P_{\frac{3 p - 7}{10}}^2)
(T^2 - b_2^2 P_{\frac{p - 9}{10}}^2)
(T - b_2 P_{\frac{p - 1}{2}})$ 
\\\hline
6&
$p \equiv 1 \pmod{12}$ &
\rule{0pt}{15pt}
\begin{tabular}[c]{@{}l@{}}$ T^6
	(T - b_{12}^{7} P_{\frac{5 p - 5}{12}})
	(T - b_{12}^{5} P_{\frac{5 p - 5}{12}})
	(T - b_{12}^{9} P_{\frac{p - 1}{4}})
	(T - b_{12}^{3} P_{\frac{p - 1}{4}})(T - b_{12}^{11} P_{\frac{p - 1}{12}}) \times$
	\\
	$\times (T - b_{12} P_{\frac{p - 1}{12}})$
\end{tabular} 
\\\hline
6&
$p \equiv 5 \pmod{12}$ &
\rule{0pt}{10pt}
$T^6
(T^2 - b_2^3 P_{\frac{5 p - 1}{12}} P_{\frac{p - 5}{12}})
(T^2 - b_2 P_{\frac{5 p - 1}{12}} P_{\frac{p - 5}{12}})
(T - b_4^3 P_{\frac{p - 1}{4}})
(T - b_4 P_{\frac{p - 1}{4}})$ 
\\\hline
6&
$p \equiv 7 \pmod{12}$ &
\rule{0pt}{10pt}
$T^6
(T^2 - b_2^2 P_{\frac{p - 3}{4}}^2)
(T^2 - b_3^4 P_{\frac{5 p - 11}{12}} P_{\frac{p - 7}{12}})
(T^2 - b_3^2 P_{\frac{5 p - 11}{12}} P_{\frac{p - 7}{12}})$ 
\\\hline
6&
$p \equiv 11 \pmod{12}$ &
\rule{0pt}{10pt}
$T^6
(T^2 - b_2^2 P_{\frac{5 p - 7}{12}}^2)
(T^2 - b_2^2 P_{\frac{p - 3}{4}}^2)
(T^2 - b_2^2 P_{\frac{p - 11}{12}}^2)
$ 
\\\hline
7&
$p \equiv 1 \pmod{7}$ &
\rule{0pt}{15pt}
\begin{tabular}[c]{@{}l@{}}$T^7
	(T - b_2 P_{\frac{p - 1}{2}})
	(T - b_{14}^9 P_{\frac{5 p - 5}{14}})
	(T - b_{14}^5 P_{\frac{5 p - 5}{14}})
	(T - b_{14}^{11} P_{\frac{3 p - 3}{14}})
	(T - b_{14}^3 P_{\frac{3 p - 3}{14}})
	\times$\\
	$\times
	(T - b_{14}^{13} P_{\frac{p - 1}{14}})
	(T - b_{14} P_{\frac{p - 1}{14}})$
\end{tabular}
\\\hline
7&
$p \equiv 2 \pmod{7}$ &
\rule{0pt}{10pt}
$T^7
(T^3 - b_2^3 P_{\frac{5 p - 3}{14}} P_{\frac{3 p - 13}{14}} P_{\frac{p - 9}{14}})^2
(T - b_2 P_{\frac{p - 1}{2}})$ 
\\\hline
7&
$p \equiv 3 \pmod{7}$ &
\rule{0pt}{10pt}
$T^7
(T^6 - b_2^6 P_{\frac{5 p - 1}{14}}^2 P_{\frac{3 p - 9}{14}}^2 P_{\frac{p - 3}{14}}^2)
(T - b_2 P_{\frac{p - 1}{2}})$ 
\\\hline
7&
$p \equiv 4 \pmod{7}$ &
\rule{0pt}{10pt}
$T^7 (T^3 - b_2^3 P_{\frac{5 p - 13}{14}} P_{\frac{3 p - 5}{14}} P_{\frac{p - 11}{14}})^2 (T - b_2 P_{\frac{p - 1}{2}})$ 
\\\hline
7&
$p \equiv 5 \pmod{7}$ &
\rule{0pt}{10pt}
$T^7
(T^6 - b_2^6 P_{\frac{5 p - 11}{14}}^2 P_{\frac{3 p - 1}{14}}^2 P_{\frac{p - 5}{14}}^2)
(T - b_2 P_{\frac{p - 1}{2}})
$ 
\\\hline
7&
$p \equiv 6 \pmod{7}$ &
\rule{0pt}{10pt}
$T^7
(T^2 - b_2^2 P_{\frac{5 p - 9}{14}}^2)
(T^2 - b_2^2 P_{\frac{3 p - 11}{14}}^2)
(T^2 - b_2^2 P_{\frac{p - 13}{14}}^2)
(T - b_2 P_{\frac{p - 1}{2}})$ 
\\\hline

\end{tabular}
\end{table}

\bibliography{mybibfile}

\appendix
\section{Genus 4. Equation for $a_{1}$.}
\label{appendix:a1_relation}
\[
a_{1}^{16}
+
c_{14} a_{1}^{14} 
+
c_{12} a_{1}^{12}
+
c_{10} a_{1}^{10} 
+
c_8 a_{1}^8
+
c_6 a_{1}^6
+
c_4 a_{1}^4
+
c_2 a_{1}^2
+
c_0 = 0.
\]
\begin{equation*}
\begin{split}
c_0 =& (128 q^4-128 a_{1,2} q^3+32 a_{1,2}^2 q^2+128 a_{3,2} q-64 a_{1,2} a_{2,2} q+\\&+16 a_{1,2}^3 q-64 a_{4,2}+16 a_{2,2}^2-8 a_{1,2}^2 a_{2,2}+a_{1,2}^4)^2,\\
c_2 =& -131072 q^7+163840 a_{1,2} q^6-32768 a_{2,2} q^5-65536 a_{1,2}^2 q^5-81920 a_{3,2} q^4
+\\&+
45056 a_{1,2} a_{2,2} q^4 +
5120 a_{1,2}^3 q^4+65536 a_{4,2} q^3+49152 a_{1,2} a_{3,2} q^3
+\\&-
16384 a_{2,2}^2 q^3-12288 a_{1,2}^2 a_{2,2} q^3
+ 2048 a_{1,2}^4 q^3-49152 a_{1,2} a_{4,2} q^2
+\\&+
4096 a_{1,2}^2 a_{3,2} q^2+8192 a_{1,2} a_{2,2}^2 q^2-5120 a_{1,2}^3 a_{2,2} q^2+768 a_{1,2}^5 q^2
+\\&+
16384 a_{2,2} a_{4,2} q-16384 a_{3,2}^2 q+4096 a_{1,2} a_{2,2} a_{3,2} q-1024 a_{1,2}^3 a_{3,2} q
+\\&-
4096 a_{2,2}^3 q + 4096 a_{1,2}^2 a_{2,2}^2 q-1280 a_{1,2}^4 a_{2,2} q+128 a_{1,2}^6 q
+\\&+
8192 a_{3,2} a_{4,2}-6144 a_{1,2} a_{2,2} a_{4,2}+1536 a_{1,2}^3 a_{4,2} + 2048 a_{2,2}^2 a_{3,2}
+\\&-
1024 a_{1,2}^2 a_{2,2} a_{3,2}+128 a_{1,2}^4 a_{3,2}-512 a_{1,2} a_{2,2}^3 +
384 a_{1,2}^3 a_{2,2}^2
+\\&-
96 a_{1,2}^5 a_{2,2}+8 a_{1,2}^7,\\
c_4 =& 
253952 q^6-233472 a_{1,2} q^5+47104 a_{2,2} q^4+65024 a_{1,2}^2 q^4 + 57344 a_{3,2} q^3
+\\&-
26624 a_{1,2} a_{2,2} q^3-5632 a_{1,2}^3 q^3-61440 a_{4,2} q^2-12288 a_{1,2} a_{3,2} q^2
+\\&+
7168 a_{2,2}^2 q^2-2048 a_{1,2}^2 a_{2,2} q^2+1344 a_{1,2}^4 q^2+22528 a_{1,2} a_{4,2} q
+\\&-
2048 a_{2,2} a_{3,2} q-2560 a_{1,2}^2 a_{3,2} q+3584 a_{1,2} a_{2,2}^2 q-1280 a_{1,2}^3 a_{2,2} q
+\\&+
96 a_{1,2}^5 q-7168 a_{2,2} a_{4,2}+1280 a_{1,2}^2 a_{4,2}+4096 a_{3,2}^2
+\\&-
2048 a_{1,2} a_{2,2} a_{3,2}+512 a_{1,2}^3 a_{3,2}-256 a_{2,2}^3+576 a_{1,2}^2 a_{2,2}^2
+\\&-240 a_{1,2}^4 a_{2,2}+28 a_{1,2}^6,\\
c_6 =&
-204800 q^5+136192 a_{1,2} q^4-16384 a_{2,2} q^3-29696 a_{1,2}^2 q^3
+\\&-
12288 a_{3,2} q^2+1024 a_{1,2} a_{2,2} q^2+4096 a_{1,2}^3 q^2 + 20480 a_{4,2} q
+\\&-
1024 a_{1,2} a_{3,2} q+1024 a_{2,2}^2 q-320 a_{1,2}^4 q-2560 a_{1,2} a_{4,2}
+\\&-
1024 a_{2,2} a_{3,2}+768 a_{1,2}^2 a_{3,2}+384 a_{1,2} a_{2,2}^2-320 a_{1,2}^3 a_{2,2}+56 a_{1,2}^5,\\
\end{split}
\end{equation*}
\begin{equation*}
\begin{split}
c_8 =& 79104 q^4-38144 a_{1,2} q^3+512 a_{2,2} q^2+7104 a_{1,2}^2 q^2+256 a_{3,2} q
+\\&+
640 a_{1,2} a_{2,2} q-800 a_{1,2}^3 q-2176 a_{4,2}+512 a_{1,2} a_{3,2}+96 a_{2,2}^2
+\\&-
240 a_{1,2}^2 a_{2,2}+70 a_{1,2}^4,\\
c_{10} =& -15360 q^3+5376 a_{1,2} q^2+256 a_{2,2} q-768 a_{1,2}^2 q+128 a_{3,2} +\\&
-96 a_{1,2} a_{2,2}+56 a_{1,2}^3,\\
c_{12} =& 1472 q^2-352 a_{1,2} q-16 a_{2,2}+28 a_{1,2}^2,\\
c_{14} =& 8 a_{1,2}-64 q.\\
\end{split}
\end{equation*}
\end{document}